\documentclass[10pt, conference,twocolumn]{ieeeconf} 

\IEEEoverridecommandlockouts    
\usepackage{psfrag}
\usepackage{amsmath,amssymb,amsfonts,bbm,amsthm}
\usepackage{latexsym}
\usepackage{graphicx}
\usepackage{footmisc}

\usepackage{caption}
\usepackage{subcaption}

\usepackage{colortbl}
\usepackage{fancyhdr}
\usepackage{epstopdf}
\usepackage{float}

\usepackage{hyperref}
\usepackage{booktabs}
\usepackage{enumerate}
\usepackage[ruled]{algorithm2e}
\usepackage{balance}
\usepackage{mathrsfs}

\usepackage[dvipsnames]{xcolor}

\usepackage{tikz}
\usepackage{pgfplots}

\usepackage{color}
\usepackage{lipsum}

\usepackage[deletedmarkup=sout,authormarkup=superscript]{changes}

\pgfplotsset{compat=1.10}
\usepgfplotslibrary{fillbetween}

\usepackage[font=small , skip = -6pt]{caption}

\newtheorem{theorem}{Theorem}

\newtheorem{lemma}{Lemma}
\newtheorem{corollary}{Corollary}

\newtheorem{remark}{Remark}
\newtheorem{assumption}{Assumption}

\IEEEoverridecommandlockouts

\newcommand{\reals}{\mathbb{R}}

\newcommand{\mto}[0]{\rightrightarrows}
\newcommand{\eigmax}[1]{{\lambda}_{\max} \left({#1}\right)}
\newcommand{\eigmin}[1]{{\lambda}_{\min}\left({#1}\right)}

\definechangesauthor[name={Giuseppe}, color=NavyBlue]{GB}
\definechangesauthor[name={Dominic}, color=RedViolet]{DLM}

\newcommand{\mc}{\mathcal}
\newcommand{\bb}{\mathbb}
\newcommand{\R}{\bb R}

\newcommand\K[0]{\mathcal{K}}

\newcommand{\argmin}{\operatorname{argmin}}

\newcommand{\proj}{\mathrm{proj}}
\newcommand{\prox}{\mathrm{prox}}

\newcommand{\nc}{\mathcal{N}}

\newcommand{\Rmnum}[1]{\expandafter\@slowromancap\romannumeral #1@}

\makeatletter
\begin{document}

\title{Sampled-Data Online Feedback Equilibrium Seeking: Stability and Tracking}

\author{Giuseppe Belgioioso,   Dominic Liao-McPherson,  Mathias Hudoba de Badyn, \\ Saverio Bolognani, John Lygeros, and Florian D\"{o}rfler
\thanks{The authors are with the Automatic Control Laboratory, ETH Z\"{u}rich, Physikstrasse 3, 8092 Z\"{u}rich, Switzerland. Emails: \texttt{\{gbelgioioso, dliaomc, mbadyn, bsaverio, jlygeros, dorfler\}@ethz.ch}. This research is supported by the SNSF through NCCR Automation. 
}
}

\pagestyle{empty}
\maketitle         
\thispagestyle{empty}

\begin{abstract} 
This paper proposes a general framework for constructing feedback controllers that drive complex dynamical systems to ``efficient" steady-state (or slowly varying) operating points. Efficiency is encoded using generalized equations which can model a broad spectrum of useful objectives, such as optimality or equilibria (e.g. Nash, Wardrop, etc.) in noncooperative games. The core idea of the proposed approach is to directly implement iterative solution (or equilibrium seeking) algorithms in closed loop with physical systems. Sufficient conditions for closed-loop stability and robustness are derived; these also serve as the first closed-loop stability results for sampled-data feedback-based optimization. Numerical simulations of smart building automation and game-theoretic robotic swarm coordination support the theoretical results.
\end{abstract}

\section{Introduction}
It is desirable to operate many engineering systems at an ``optimal'' steady-state, for example: congestion control in communication networks \cite{low2002internet}, supply/demand balancing in power systems \cite{wood2013power}, and temperature control in smart buildings \cite{hatanaka2017integrated}. However, while optimality is a suitable solution concept in countless settings, game-theoretic solution concepts, such as (generalized) Nash Equilibria (NE) \cite{facchinei2010generalized}, are more appropriate tools for modelling  multi-agent systems with possibly self-interested agents that interact through a shared dynamical system. This is the case for many modern infrastructures such as power systems, e.g. in energy markets \cite{saad2012game} and electric vehicle fleets \cite{ma2011decentralized}, and traffic  networks \cite{barrera2014dynamic}.

A typical paradigm for optimal system operation is to design a tracking-type controller and periodically solve an optimization problem to compute the system setpoints. We will refer to this process as \textit{feedforward optimization}. In contrast, \textit{feedback optimization} (FO) seeks to design controllers that steer the plant to the (a-priori unknown) solution of an optimization problem \cite{simonetto2020time, hauswirthoptimization}. The argument for using feedback over feedforward optimization is the same as for feedback control: superior robustness to noise, model uncertainty, and unmeasured disturbances.

A convex optimization-based method for FO of static networks is presented in \cite{bernstein2019online} along with regret and tracking error analyses. Methods for non-convex problems are presented in \cite{haberle2020non,tang2018feedback} along with convergence and tracking results. A projected dynamical systems approach is explored in \cite{hauswirth2018time}. The robustness of FO is investigated theoretically \cite{colombino2019towards} and experimentally \cite{ortmann2020experimental}. All these works assume that the plant is a static map and thus neglect dynamical interactions.

In reality, the plant is often a dynamical system, a fact that can lead to stability issues \cite[Ex.~1.2]{hauswirthoptimization}. Asymptotic stability results for linear time invariant systems (LTI) coupled with saddle-flow dynamics are provided in \cite{colombino2019online}. The interconnection of input-constrained nonlinear systems with gradient flow dynamics is studied in \cite{hauswirth2020timescale}. A design framework for unconstrained online optimization of LTI systems is presented in \cite{lawrence2020linear} and related work on low-gain integral controllers for nonlinear systems can be found in \cite{simpson2020analysis}. Finally, the stability of saddle-flow-based FO for a class of nonlinear systems to a constrained convex optimization problem is investigated in \cite{li2020optimal}. Notably, all of these works use continuous-time flows.

FO uses static models and is primarily concerned with steady-state or slowly varying optimization. This contrasts with both suboptimal \cite{diehl2005nominal,liao2020time} and economic \cite{ellis2014tutorial} model predictive control  which uses a dynamic model to optimize trajectories, and extremum seeking (ES) \cite{ariyur2003real}, which does not use any model information.

In the context of games with dynamical agents, two main control approaches have been considered, built on passivity-based \cite{gadjov2018passivity,romano2019dynamic} and ES algorithms \cite{stankovic2011distributed, krilavsevic2021learning}, respectively.
In \cite{gadjov2018passivity}, passivity is leveraged to design a distributed control law with convergence guarantees to a NE in games with single-integrator agents. The extension to the case of multi-integrator agents with dynamics affected by (partially-known) disturbances is presented in \cite{romano2019dynamic}.
In \cite{stankovic2011distributed}, an ES-based controller is designed for NE seeking in games with nonlinear dynamical agents. The extension to games with coupling constraints is presented for the first time in \cite{krilavsevic2021learning}. In both cases, convergence is proven only in the disturbance-free case, and to an approximate solution of the game.

This paper extends feedback optimization to \textit{feedback equilibrium seeking} (FES), wherein the goal is to design a feedback controller that drives the plant to the solution trajectory of a time-varying generalized equation (GE).
GEs contain constrained optimization as a special case but can model a broader range of phenomena including Nash and Wardrop equilibiria. Our contributions are threefold:
\begin{enumerate}[1.]
\item We introduce a general design framework for feedback equilibrium seeking in nonlinear systems.
\item We provide a closed-loop stability analysis of the interconnection between the algorithm and the physical system in a sampled-data setting, derive sufficient conditions for input-to-state stability of the closed-loop system, and characterize the asymptotic tracking error. The analysis is also valid for discrete-time plants.
\item We validate and illustrate our results using smart building control and distributed robotic coordination examples.
\end{enumerate}
Notably, FES contains FO as a special case. As such, this paper provides the first closed-loop stability analysis of feedback optimization in a sampled-data setting (which is extremely common in practice).

Our sampled-data approach enables the use of well-known algorithms from optimization and (monotone) operator theory, e.g. splitting methods \cite[\S~26]{bauschke2011convex}. Moreover, it facilitates distributed implementations, since communication in multi-agent systems is naturally a discrete process.


\medskip
\textit{Notation:}
Given a positive definite matrix $P\!=\!P^\top$, $\|x\|_P = \sqrt{x^T P x}$, it is understood that $\|x\| \!= \!\|x\|_I$. The largest (smallest) eigenvalue of $P$ is denoted by $\eigmax{P}$, ($\eigmin{P}$), and the condition number by $\kappa(P)\! =\! \frac{\eigmax{P}}{\eigmin{P}}$. Given a closed convex set $\Omega \subseteq \reals^n$, $\iota_{\Omega}:\reals^n \to \{0,\infty\}$ denotes its indicator function, $\nc_{ \Omega}(x):\Omega \mto \reals^n$ denotes its normal cone operator, and $\proj_{\Omega}:\reals^n \to \Omega$ is the Euclidean projection onto $\Omega$. A set-valued mapping $\mc B:\reals^n \mto \reals^n$ is $\mu$-strongly monotone, with $\mu >0$, if $(u-v)^\top (x-y) \geq \mu \left\| x-y \right\|^2$ for all $x \neq y \in \R^n$, $u \in \mathcal{B} (x)$, $v \in \mathcal{B} (y)$, and monotone if $\mu = 0$. For a convex function $f:\reals^n \to \reals$, $\partial f:\reals^n \mto \reals^n$ denotes the subdifferential mapping in the sense of convex analysis; $\prox_{ f}$ its proximal operator, s.t. $\prox_{ f}(x) \!=\! \argmin_{\xi}  f(\xi)+\frac{1}{2} \|x-\xi \|^2$ $\forall x \in \R^n$. A function $\gamma:\R_{\geq 0} \rightarrow \R_{\geq 0} $ is a $\mc K$-function if it is continuous, strictly increasing and $\gamma(0)=0$. We denote signals by $\mc{L}^n = \{f: \reals_{\geq0} \to \reals^n\}$. Given $\xi\in \reals^n$ and $f:\reals^n \to \reals^n$, a signal $x\in \mc{L}^n$ is a (Carathéodory) solution of the initial value problem (IVP) $\dot x(t) = f(x(t))$, $x(0) = \xi $ if it is uniformly continuous, satisfies the initial condition and the differential equation almost everywhere (in time). 


\section{Problem Setting}
In this paper, we consider the following nonlinear continuous-time state space system
\begin{subequations} \label{eq:dynamics}
\begin{align}
\dot x(t) &= f(x(t),u(t),w(t))\\
y(t) &= g(x(t),w(t))
\end{align}
\end{subequations}
where $x\in \mc{L}^{n_x}$ is the state, $u\in \mc{L}^{n_u}$ the control input, $w\in \mc{L}^{n_w}$ an exogenous disturbance, and $y\in \mc{L}^{n_y}$ the output.

\begin{assumption} \label{ass:sys}
(i) $f$ is locally Lipschitz continuous, and $g$ is $\ell_{\text g}$-globally Lipschitz continuous w.r.t. its first argument. (ii) For all $ t\geq 0$, $u(t)\in \mc{U}$ where $\mc{U}\subset \reals^{n_u}$ is compact and convex. (iii) For all $ t\geq 0$, $w(t) \in \mc{W}$ where $\mc{W}\subset \reals^{n_w}$ is compact (iv) $\dot{w}\in \mc{L}^{n_w}$ is essentially bounded.
{\hfill $\square$}
\end{assumption}
We assume that \eqref{eq:dynamics} is pre-stabilized in the following sense.
\begin{assumption} \label{ass:ctime_lyapunov}
There exists a continuously differentiable steady-state mapping $x_{\text{ss}}:\mc{U}\times \mc{W}\to \reals^{n_x}$, such that $\forall u\in \mc{U},w\in \mc{W}$, $f(x_{\text{ss}}(u,w),u,w) = 0$. Moreover, there exists a continuously differentiable function $V:\reals^{n_x}\times\mc{U}\times\mc{W} \to \reals$, and constants $\mu,\alpha_1,\alpha_2> 0$, and a $\mc K$-function $\sigma_{\text c}$ such that:
\begin{enumerate}[(i)]
\item for all $x\in \reals^{n_x},u\in \mc{U}$ and $w\in \mc{W}$, 
\begin{equation*}
    \alpha_1 \|x - x_{\text{ss}}(u,w)\|^2 \leq V(x,u,w)\leq \alpha_2 \|x-x_{\text{ss}}(u,w)\|^2
\end{equation*}
\item For any constant $u\in \mc{U}$, $\dot{V}(t) \leq -\mu V(t) + \sigma_{\text c}(\|\dot w(t)\|)$, where $V(t) = V(x(t),u,w(t))$ and $x$ satisfies \eqref{eq:dynamics}. {\hfill $\square$}
\end{enumerate}
\end{assumption}
Under Assumption \ref{ass:ctime_lyapunov}, a steady-state input-output mapping for the system \eqref{eq:dynamics} exists and we denote it by
\begin{equation} \label{eq:IOssMapping}
    h(u,w) = g(x_{\text{ss}} (u,w),w).
\end{equation}
Finally, we make the following assumptions about the availability of measurements.
\begin{assumption}
The output $y$ is measured and the exogenous disturbance $w$ in unmeasured.
{\hfill $\square$}
\end{assumption}

\begin{remark}
The inability to measure $w$ is common in areas such as power systems, where $w$ represents variable microgeneration and loads caused by consumers drawing power from the grid. In other cases, $w$ may be partially measured. For example, in the smart buildings example in Section~\ref{ss:numerical_example1}, the ambient temperature is measured while the solar heat flux and the heat released by building occupants are not.
Measured portions of $w$ should be included in $y$.
\end{remark}

\begin{figure}[t]
    \centering
    \includegraphics[width=0.8\columnwidth]{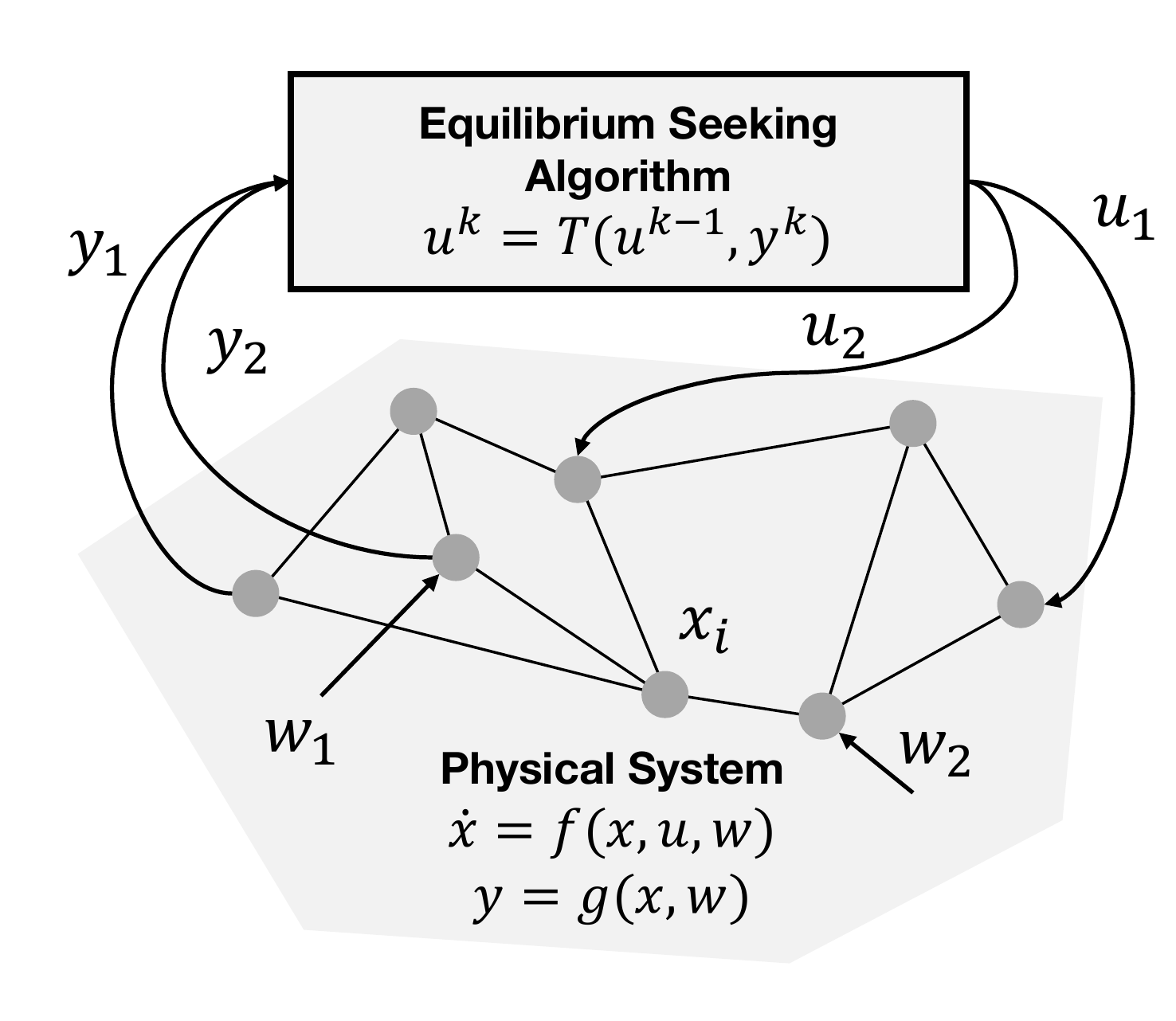}
    \caption{In feedback equilibrium seeking, measurements from a dynamical system are incorporated into an equilibrium seeking algorithm resulting in a coupled sampled-data system.}
    \label{fig:block-diagram}
\end{figure}

\smallskip
Our objective is designing a feedback controller that will drive \eqref{eq:dynamics} to ``efficient'' operating conditions. We use the following parameterized GE to encode efficiency:
\begin{subequations}
\label{eq:GE}
\begin{align} 
 0 &\in F(u,y) + \mc B(u),\\
 y &= h(u,w),
\end{align}
\end{subequations}
where $F: \R^{n_u } \!\times\! \R^{n_y}  \!\to\! \reals^{n_u}$ is continuously differentiable and $\mc{B}:\mc{U}\mto \reals^{n_u}$ is a set-valued mapping. 

GEs can be used to encode a variety of useful objectives, notably constrained optimization problems or (generalized) Nash games. Two relevant examples are provided in Sections~\ref{ex:opt_zeros} and \ref{ex:game_zeros} below.

Mathematically, we formalize the problem as follows: Design an output feedback controller that will drive \eqref{eq:dynamics} to solution trajectories $y^*(w(t))$ and $u^*(w(t))$ of \eqref{eq:GE} where
\begin{equation} \label{eq:solution_mapping}
    \begin{bmatrix}
        u^*(w)\\ y^*(w)
    \end{bmatrix} = \left \{\begin{matrix}
      u \in \mc{U}  ,\\ y\in \reals^{n_y}	
    \end{matrix} ~\bigg|~
    \begin{matrix}
        F(u,y)+ \mc{B}(u)\ni 0,\\ y = h(u,w)
    \end{matrix} \right\}
\end{equation}
is the parameter-to-solution mapping. To ensure this problem is well posed, we impose the following conditions on \eqref{eq:GE}. 
\begin{assumption} \label{ass:existence}
The following hold:
\begin{enumerate}[(i)]
\item $u^*$ and $y^*$ are functions (i.e., single valued);
\item $\exists \ell_{u^*} > 0$ such that $\|u^*(w')- u^*(w)\| \leq \ell_{u^*} \|w'-w\|$ for all $w, w' \in \mc{W}$. {\hfill $\square$}
\end{enumerate}  
\end{assumption}
Conditions under which Assumption~\ref{ass:existence} holds are provided on a per example basis.

\subsection{Non-Smooth Convex Feedback Optimization} \label{ex:opt_zeros}
Consider the parameterized optimization problem:
\begin{subequations} \label{eq:OPT}
\begin{align} \label{eq:CF}
    \min_{u,y}   &\quad  \varphi_1(u,y) + \varphi_2(u) \\
    \text{s.t.} & \quad  y = h(u,w), \label{eq:Constr-1}
\end{align}
\end{subequations}
where $\varphi_1$ and $h$ are continuously differentiable, $h(u,w) = h_u(u) + h_w(w)$, and $\varphi_2:\reals^{n_u}\to \reals \cup \{\infty\}$ is a proper lower semicontinuous convex function (e.g. $\varphi_2 = \iota_{\mc U}$, with $\mc U $ closed convex). By substituting \eqref{eq:Constr-1} into \eqref{eq:CF}, we get 
\begin{align} \label{eq:UOP}
    \min_u  &\quad \varphi_1(u,h(u,w))+  \varphi_2(u).
\end{align}
If $\varphi_1(u,h(u,w))$ is convex in $u$, the following parametrized GE is necessary and sufficient for optimality of $\bar u \in \reals^{n_u}$
\begin{align} \label{eq:OFOge}
   0 & \in F (\bar u,h(\bar u,w)) + \partial \varphi(\bar u),
\end{align}
where $F(u,y) = \nabla_u \varphi_1(u,y) + \nabla{h_u(u)}^\top \nabla_y \varphi_1(u,y)$.

Further, strong convexity of $\varphi_1(u,h(u,w))$ with respect to $u$ is sufficient for Assumption~\ref{ass:existence}, see e.g., \cite[Th.~ 2B.1]{dontchev2009implicit}.

\subsection{Feedback NE Seeking in Monotone Games} \label{ex:game_zeros} 
Consider a set of agents (i.e., players), $\mc I = \{1,\ldots,N \}$, where each agent $i \in \mc I$ shall choose a control input (i.e., strategy or action) $u_i \in \mc U_i$, where $\mc U_i\subseteq \R^{n_i}$ is a closed convex strategy set. Often, see e.g., \cite{stankovic2011distributed}, the dynamics of the agents are decoupled and \eqref{eq:dynamics} can be expressed as
\begin{equation}
   \forall i\in \mc I:
   \begin{cases}
    \dot{x}_i(t) = f_i(x_i(t),u_i(t),w_i(t)),\\
    y_i(t) = g_i(x_i(t),w_i(t)),
    \end{cases}
\end{equation}
inducing the local steady-state mappings $h_i(u_i,w_i)$. Moreover, each agent $i$ has a cost function $J_i(u_i,h(u,w))$ that depends on its own action and the actions of other agents\footnote{We adopt the game-theoretic notation $u\! =\! [u_i]_{i\in \mc{I}}\! :=\! [u_1^T, \ldots, u_N^T]^T$, $u_{-i} \!:= \![u_j]_{j\in \mc{I}\setminus \{i\}}$, $(u_i, u_{-i}) \!:=\! u$; further, $h(u,w)\!=\![h_i(u_i,w_i)]_{i \in \mc I}$.}.

Each agent is self-interested and wants to optimize its own objective; the resulting collection of the $N$ parametrized inter-dependent optimization problems constitutes a game:
\begin{align} \label{eq:Game}
\forall i \in \mc I: 
\begin{array}{r l} \displaystyle
  \min_{u_i \in \mc{U}_i}   &  \left\{J_i(u_i,y)~\big |~y = h(u,w)\right\}.
\end{array}
\end{align} 
From a game-theoretic perspective, a relevant solution concept for \eqref{eq:Game} is the Nash equilibrium (NE), where no agent can unilaterally reduce its cost, see e.g. \cite[\S~1]{scutari2014real}.

If each $J_i(u_i,h(u,w))$ is convex and continuously differentiable in $u_i$, a strategy profile $\bar u = [\bar u_i]_{i\in \mc{I}}$ is a NE of the game in \eqref{eq:Game} if and only if it is a solution to the following parametrized GE (or variational inequality) \cite[Cor.~3.4]{facchinei2010generalized}:
\begin{align} \label{eq:NEge}
0 \in F(\bar u,h(\bar u,w)) + \mc N_{\mc U}(\bar u),
\end{align}
where $\mc N_{\mc U} = \partial \iota_{\mc U}$, $\mc U := \prod_{i \in \mc I} \mc U_i$, and $F = [F_i]_{i\in \mc{I}}$ is the so-called pseudo-gradient mapping, whose components are
\begin{align*} 
F_i(u_i,y) := \nabla_{u_i} J_i(u_i,y) + \nabla_{u_i} {h_i(u_i,w_i)}^\top \nabla_{y_i} J_i(u_i,y).
\end{align*}
In this context, a sufficient condition for Assumption~\ref{ass:existence} to hold is strong monotonicity of the pseudo-gradient $F$ in $u$, and Lipschitz continuity of $F$ w.r.t. $w$ \cite[Theorem 2F.6]{dontchev2009implicit}.

\section{Control Strategy}
Our objective is to maintain \eqref{eq:dynamics} near efficient operating points, as defined in \eqref{eq:solution_mapping}. If $w$ were measurable, then  picking $u(t) = u^*(w(t))$ would cause \eqref{eq:dynamics} to track the solution trajectory $y^*$ with bounded error. However, $u^*$ is rarely available in closed form and in practice $u^*(w)$ is computed numerically using an equilibrium seeking algorithm. In this paper, we illustrate how to construct a feedback controller by incorporating measurements into these algorithms. 

Consider a class of abstract algorithms for solving \eqref{eq:GE}
\begin{align} \label{eq:FPI-opt}
\begin{array}{l}
 u^{k+1} =  T(u^k, h(u^k,w)),
\end{array}
\quad (\forall k \in \bb N)
\end{align}
where $T: \mc \R^{n_u} \times \R^{n_y} \rightarrow \R^{n_u}$ is the rule for generating the next iterate. We assume that $\eqref{eq:FPI-opt}$ converges linearly to $u^*(w)$, whenever $w$ is 	measured and held constant, and that $T$ is well-behaved in a parameterized setting. 
\begin{assumption}\label{ass:optRoutine}
Define $\tilde T(u,w):= T(u,h(u,w))$; then, for all fixed $w \in \mc W $, the following holds: 
\begin{enumerate}[(i)]
    \item $u^*(w) = \tilde T(u^*(w),w)$;
    \item $\exists P=P^\top \succ 0$ and $c_{\text T} \in (0,1)$ such that, for all $ u \in \mc{U}$,  $\|\tilde T(u,w) - u^*(w) \|_P \leq c_{\text T} \| u-u^*(w)\|_P$.
    {\hfill $\square$}
\end{enumerate}
\end{assumption} 
\begin{assumption} \label{ass:T_lipschitz}
For all $ u \in \mc \R^{n_u}$, there exists a constant $\ell_{\text T}$ s.t. $\|T(u,y)-T(u,y')\| \leq \ell_{\text T} \| y-y'\|$, $\forall y,y' \in \mathbb R^{n_y}$.
{\hfill $\square$}
\end{assumption}

In our problem setting, $w$ is not measurable (nor constant), so we run the algorithm \eqref{eq:FPI-opt} in parallel with the plant and replace evaluations of the steady-state map $h$ with measurements of $y$ obtained from the system. This \textit{online-feedback-equilibrium-seeking} strategy renders the optimization routine robust to unmeasured $w$ and to modelling errors\footnote{While we do not explicitly address the issue here, feedback optimization is known to be robust to modelling errors, see e.g., \cite{colombino2019towards,ortmann2020experimental}.}.

Since our algorithm is discrete, we adopt a sampled-data strategy combined with a zero-order hold. Let $\tau \!> \!0$ be the sampling period and let $t^{k} = k\tau$ denote the sampling instants. This results in the sampled-data closed-loop system%
\begin{subequations} \label{eq:sampled-data-system}
\begin{align} 
\Sigma_1^s:&\begin{cases} \label{eq:sys1}
    ~\dot x(t) = f(x(t),u(t),w(t)),\\
    ~y(t) = g(x(t),w(t)).
\end{cases}\\
\label{eq:hybrid1}
\Sigma_2^s: &\begin{cases} 
~u^{k} = (1\!-\!\varepsilon)u^{k-1}+\varepsilon \,T(u^{k-1},y(t^{k})),\\
    ~u(t) = u^k,~~\forall t\in [t^k,t^{k+1}),
\end{cases}
\end{align}
\end{subequations}
where $\varepsilon \in (0,1]$ is a relaxation parameter used to regulate the control action generated by \eqref{eq:FPI-opt}. We next provide two concrete examples of equilibrium seeking algorithms before proceeding to a closed-loop stability analysis.

\begin{remark}
Using algorithms (as opposed to continuous flows) is especially important in distributed systems where each iteration requires communications between agents.
\end{remark}

\begin{remark}
Implementing FES only requires the ability to measure $y$ and evaluate $T$. In Theorem~\ref{thm:sampled_data_stability}, we derive a small-gain type stability condition, formally checking it requires additional side information about the algorithm (e.g., convergence rate $c_{\text T}$,  Lipschitz constant $\ell_{\text T}$) and the physical system (e.g., Lyapunov decay $\mu$, etc.).
\end{remark}

\subsection{Online FO using Proximal-Gradient Descent} \label{ex:gradient_descent}
Consider the optimization problem in Section~\ref{ex:opt_zeros}. Assume that $\varphi_1(u,h(u,w))$ is $m$-strongly convex with respect to $u$, and $F$ is globally $\ell$-Lipschitz continuous. Then, one algorithm satisfying Assumptions~\ref{ass:optRoutine} and \ref{ass:T_lipschitz} is the proximal-gradient method (see e.g., \cite[Prop.~26.16]{bauschke2011convex}):
\begin{align} \textstyle \label{eq:PGM}
 T(u,y) = \textup{prox}_{\gamma \varphi_2}
    \left(
    u - \gamma F(u,y)
    \right),
\end{align}
where $ \gamma \in (0, 2m/\ell^2)$ is the step size. In particular, Assumption~\ref{ass:optRoutine}~(ii) is satisfied with $c_{\text T} = \sqrt{1-\gamma(2m - \gamma \ell^2)}$.

\subsection{Online Feedback NE seeking using Best Responses} \label{ex:best_response}
Consider the game in Section \ref{ex:game_zeros}. We assume that each agent measures its own disturbance $w_i$ (i.e., $w_i$ is included in $y_i$) but does not share this information with other agents\footnote{Such information locality is often typical for self-interested agents that do not readily share private information.}. Further, since the local agents' dynamics are decoupled, $h = [h_i]_{i\in \mc{I}}$, it is possible to define functions $\tilde J_i$ such that
\begin{equation}
\tilde J_i(u_i,y_i,y_{-i}) := J_i \big( u_i,(h_i(u_i,w_i),y_{-i}) \big).
\end{equation}

If $\tilde J_i$ is strongly convex in $u_i$, a suitable algorithm is each agent $i \in \mc I$ playing the best response:
\begin{align} \label{eq:br-dyn} \displaystyle
\forall i \in \mc I: \quad T_i(w_i,y) = \underset{\xi \in \mc U_i}{\argmin}\; \tilde J_i(\xi,y_i,y_{-i}), 
\end{align}
resulting in an overall algorithm mapping $T = [T_i]_{i\in \mc{I}}$. If the game primitives satisfy the technical condition in \cite[Def.~8~(iii)]{scutari2014real}, then the best-response dynamics \eqref{eq:br-dyn} converge to a NE of \eqref{eq:Game} \cite[Th.~8]{scutari2014real} with linear rate \cite[Rem.~13]{scutari2014real}. Further, if the gradient mapping $\nabla_{\xi} \tilde{J}_i(\xi,y_i,y_{-i})$ is Lipschitz continuous with respect to $y_i$ and $y_{-i}$, then $T_i$ is Lipschitz and Assumption~\ref{ass:T_lipschitz} is satisfied \cite[2B.1]{dontchev2009implicit}.

\section{Stability and Tracking Analysis}

To analyze the closed loop \eqref{eq:sampled-data-system}, we sample the continuous-time system  \eqref{eq:sys1} to form the following discrete-time system%
\begin{subequations}  \label{eq:discrete-time-system}
\begin{align}
    \Sigma_1^d:&\begin{cases}   \label{eq:sys2d}
    ~ x^{k+1} = \psi(t^k,x^k,u^{k},w),\\
    ~ y^k = g(x^k,w^k),
    \end{cases}\\
\Sigma_2^d: &\begin{cases} \label{eq:sys1d}
    ~ u^{k+1} = (1\!-\!\varepsilon)u^k + \varepsilon \,T(u^k,y^{k+1}),
    \end{cases}
\end{align}
\end{subequations}
where $\psi:\reals \times \reals^{n_x} \times \reals^{n_u} \times \mc{L}^{n_w}\to \reals^{n_x}$ and $\psi(t^0,\xi^0,v,\omega)$ denotes the solution of the IVP
\begin{equation}
	\dot{\xi}(t) = f(\xi(t),v,\omega(t)),~~\xi(t^0) = \xi^0
\end{equation}
at time $t = t_0 + \tau$, where $\tau$ is the sampling period.
We begin by establishing the properties of the sampled plant.

\begin{lemma} \label{th:W}
Let Assumption~\ref{ass:ctime_lyapunov} hold. Given a sampling period $\tau > 0$, the system $x^{k+1} = \psi(t^k,x^k,u^{k},w)$ satisfies
\begin{equation}
    W^{k+1} \leq c_{\text W} W^k + c_{\text W} \ell_{\text W}\|u^{k+1} - u^{k}\| + \sqrt{\tau}\sigma(z^{k}),
\end{equation}
where $W \!:=\! \sqrt{V}$,  $W^k \!:=\! W(x^{k},u^{k},w^{k})$, $c_{\text W} \!:=\! \sqrt{\frac{\alpha_2}{\alpha_1}} e^{-\frac{\tau\mu}{2}}\!$, $\ell_{\text W}:=\sqrt{\alpha_1}\ell_{\text x}$ (with $\ell_{\text x}$ Lipschitz constant of $x_{\text{ss}}$ w.r.t. $u$), $\sigma := \sqrt{\sigma_{\text c}}$ is a $\mc K$-function,
and $z^k := \sup_{t\in [t^k,t^{k+1}]} \|\dot w(t)\|$.
\end{lemma}
\begin{proof}
See appendix~\ref{proof:W}.
\end{proof}

Next, we show that, for an appropriate choice of the sampling time $\tau$ and the relaxation parameter $\varepsilon$, the discrete-time system \eqref{eq:discrete-time-system} is input-to-state stable (ISS) with respect to $\dot{w}$. Moreover, we derive an explicit expression for the asymptotic input-output gain, which characterizes the tracking performances of the system.

We begin with the following preparatory lemma.
\begin{lemma} \label{lem:ParamSel}
The matrix
\begin{align}
\label{eq:M}
  M:= 
  \left[
 \begin{smallmatrix}
1\!-\!\varepsilon(1\!-\!c_{\text T}) & \frac{\|P\| \ell_{\text T}\ell_{\text g} }{\sqrt{\alpha_1}}\, \varepsilon c_{\text W}  \\
  \|P^{-1}\|  (1\!+\!c_{\text T})  \varepsilon \ell_{\text W}  c_{\text W}
&  \left( 1+  \frac{ \ell_{\text W} \ell_{\text T} \ell_{\text g}}{\sqrt{\alpha_1}} \varepsilon c_{\text W} \right) c_{\text W}
\end{smallmatrix}
\right],
\end{align}
where $c_{\text W} \!:=\! \sqrt{\frac{\alpha_2}{\alpha_1}} \, e^{-\frac{\tau\mu}{2}}\!$, has spectral radius $\rho(M)<1$ if
\begin{equation} \label{eq:small_gain}
    %
    \tau \in (\underline{\tau}, \infty), \quad
    \textstyle
    \varepsilon \in \left(
    0, \min \left\{\overline{\varepsilon}(\tau), \ 1
    \right\}
    \right],
\end{equation}
with $\underline{\tau}=\frac{\log(\alpha_2/\alpha_1)}{\mu} $ and $\overline \varepsilon(\tau)=
    \frac{  e^{\frac{\tau \mu}{2}}
    \left( e^{\frac{\tau \mu}{2}}- \sqrt{\alpha_2/\alpha_1} \right) }{ \ell_{\text g} \ell_{\text x} \ell_{\text T}  \left( 1 + \kappa(P) \frac{(1+c_{\text T})}{(1-c_{\text T})}\right)\alpha_2/\alpha_1 }$.
\end{lemma}
\begin{proof}
See Appendix~\ref{proof:ParamSel}.
\end{proof}

Now, we are ready to establish conditions for ISS of \eqref{eq:discrete-time-system}.
\begin{theorem} \label{th:ATE}
Let Assumptions~\ref{ass:sys}--\ref{ass:T_lipschitz} hold, and assume that $\tau$ and $ \varepsilon$ satisfy \eqref{eq:small_gain}. Then, the system \eqref{eq:discrete-time-system} is ISS \cite[Def.~3.1]{jiang2001input} w.r.t. $ z^k = {\sup}_{t \in [t^k,t^{k+1}]}  \|\dot w(t)\|$, i.e., there exist constants $\eta_1, \eta_2 \in \mathbb{R}_+$ and $c_{\text M} \in [0,1)$ such that
\begin{align} 
        \left\| \left[ 
 \begin{matrix}
        \delta x^{k}\\
        \delta u^{k}
    \end{matrix} 
    \right]
    \right\|
    &\leq  
    \eta_1 \, (c_{\text M})^k  \left\| \left[ \begin{matrix}
       \delta x^0\\
        \delta u^0
     \end{matrix} 
    \right] \right\| + \gamma \left( \sup_{0 \leq s \leq k-\!1} \|z^s\| \right),
    \label{eq:thISS}
\end{align}
where $\delta x^k := x^{k}-x_{\text{ss}}(u^{k}, w^{k})$, $\delta u^k := u^k - u^*(z^k)$, 
and $\gamma $ is a $\mc{K}$-function defined as
\begin{equation} 
\label{eq:gamma}
    \gamma(\zeta) := \eta_2 \ \frac{c_{\text M}}{1- c_{\text M}}  \ \left\| \left[
    \begin{matrix}
   (\ell_{u^*} \tau) \zeta \\
    (\sqrt{\tau}/c_{\text W})\, \sigma(\zeta)
    \end{matrix}
    \right] \right\|.
\end{equation}
\end{theorem}
\begin{proof}
See Appendix \ref{proof:ATE}.
\end{proof}

Next, we derive the asymptotic input-output gain of \eqref{eq:discrete-time-system}.
\begin{corollary} \label{cor:IOg}
Define the output error $\delta y^k  := y^k-y^*(w^k)$. If the preconditions of Theorem \ref{th:ATE} hold, then $\delta y^k$ satisfies
\begin{equation} \label{eq:AIOG}
    \limsup_{k\to\infty} \|\delta y^k\| \leq \gamma_{\text a}\left(\limsup_{k\to\infty}\|z^k\|\right),
\end{equation}
where the asymptotic input-output gain $\gamma_{\text a}\in \mc{K}$ is given by
$ \gamma_{\text a}(\zeta) :=   \ell_{g}(1+\ell_{\text x} ) \, \gamma(\zeta)$,
with $\gamma \in \mc K$ as in \eqref{eq:gamma}.
\end{corollary}
\begin{proof}
See Appendix \ref{proof:CorIOg}
\end{proof}

It follows immediately from  \cite[Theorem 5]{nevsic1999formulas} that ISS of the discrete-time system \eqref{eq:discrete-time-system} implies ISS of the original sampled-data system \eqref{eq:sampled-data-system}.

\begin{theorem} \label{thm:sampled_data_stability}
Let Assumptions~\ref{ass:sys}-\ref{ass:T_lipschitz} hold, and assume that $\tau$ and $ \varepsilon$ satisfy the conditions in \eqref{eq:small_gain}. Then, \eqref{eq:sampled-data-system} is input-to-state stable \cite[Definition 2.1]{sontag1995characterizations} with respect to $\dot{w}$ with the asymptotic gain between $y$ and $\dot{w}$ given by \eqref{eq:AIOG}.
{\hfill $\square$}
\end{theorem}

\begin{remark}
Theorem~\ref{thm:sampled_data_stability} holds under \eqref{eq:small_gain} which is a small-gain type condition requiring the physical system to be sufficiently fast or the algorithm to be sufficiently slow. It is always possible to make the interconnection ISS by letting $\varepsilon\to 0$, $\tau \to \infty$, or by a combination of the two. The condition \eqref{eq:small_gain} is likely conservative due to the generality of the setting we consider. Tailoring the analysis to specific classes of systems (e.g., linear dynamics) and/or algorithms will yield sharper bounds and is a topic for future work.
\end{remark}

\section{Numerical Examples}
To illustrate the utility of our framework, we present two numerical examples involving  smart buildings and robotics swarms.
Code (and parameters) for the examples can be accessed at \href{https://gitlab.nccr-automation.ch/mbadyn/fes-cdc-examples}{this gitlab link}\footnote{https://gitlab.nccr-automation.ch/mbadyn/fes-cdc-examples\label{gitlab}}~\cite{oursoftware}.

\subsection{FES of Smart Buildings} \label{ss:numerical_example1}

\begin{figure}[t]
  \centering
  \includegraphics[width=\columnwidth]{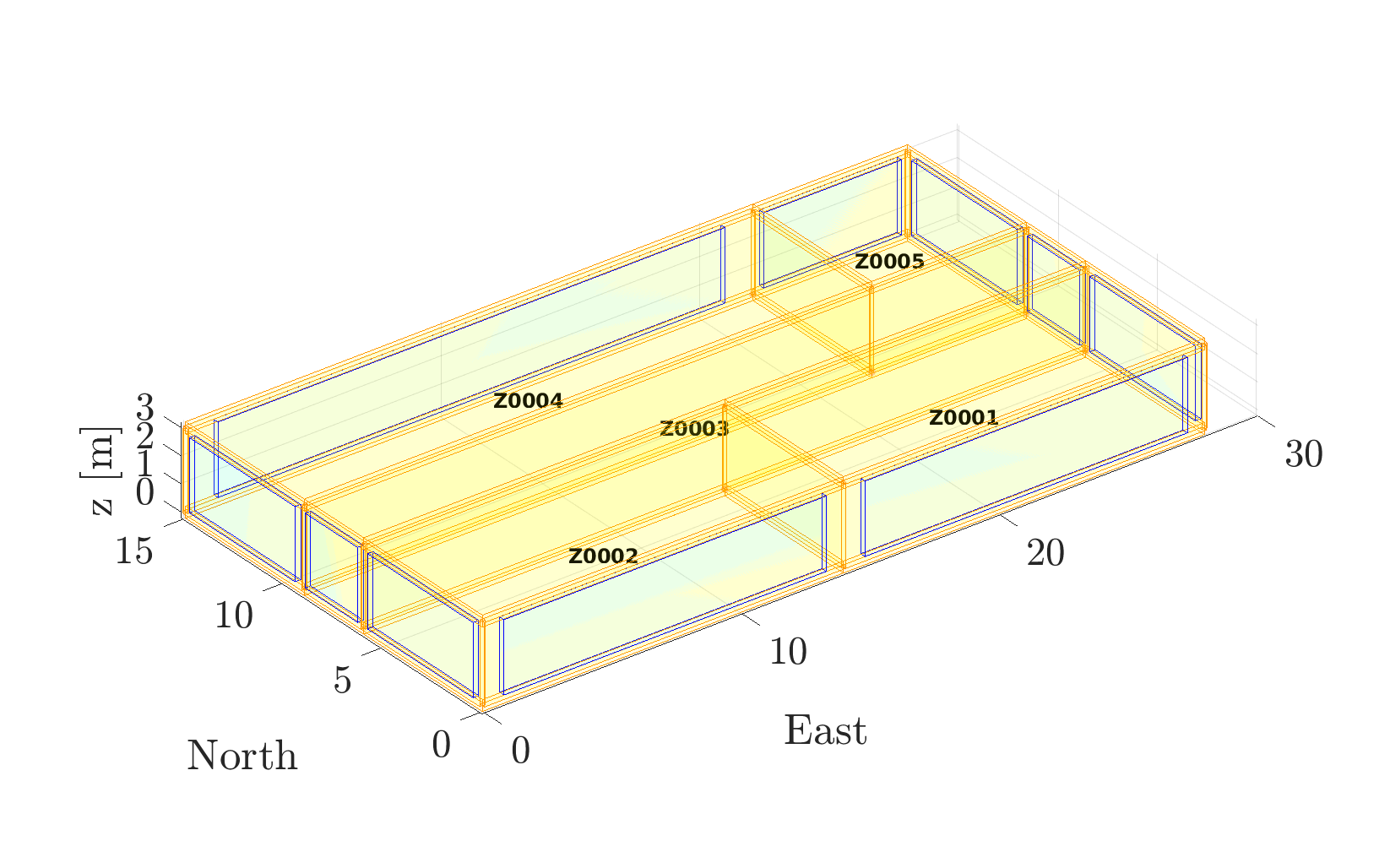}
  \caption{Example building generated via the BRCM toolbox~\protect\cite{sturzenegger2014brcm}.}
  \label{fig:building}
\end{figure}
\begin{figure}[t]
  \centering
    \includegraphics[width=\columnwidth]{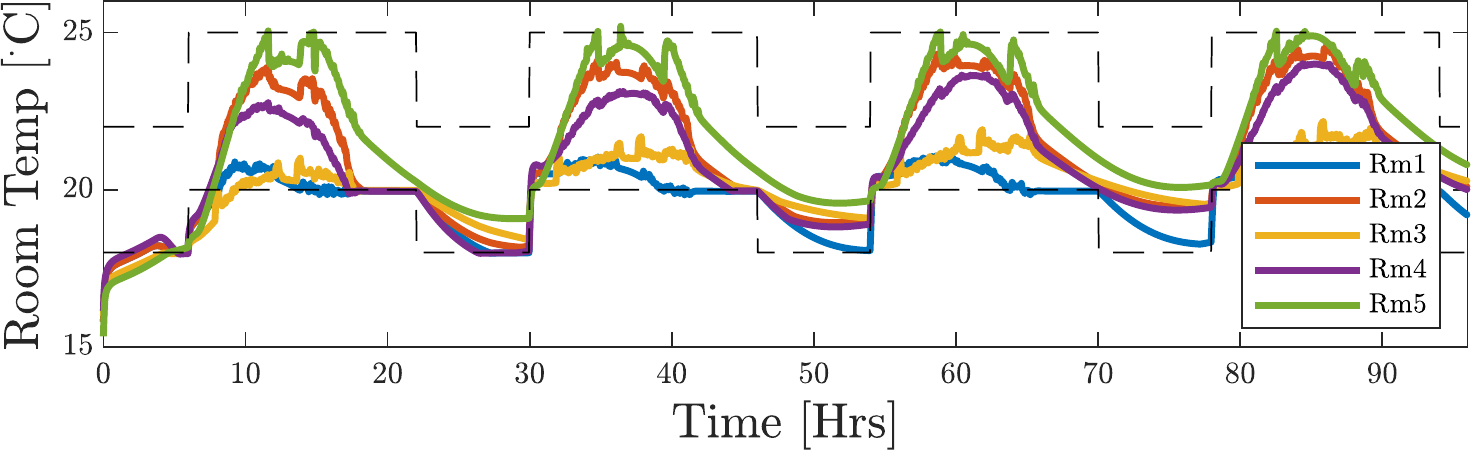}
    \includegraphics[width=\columnwidth]{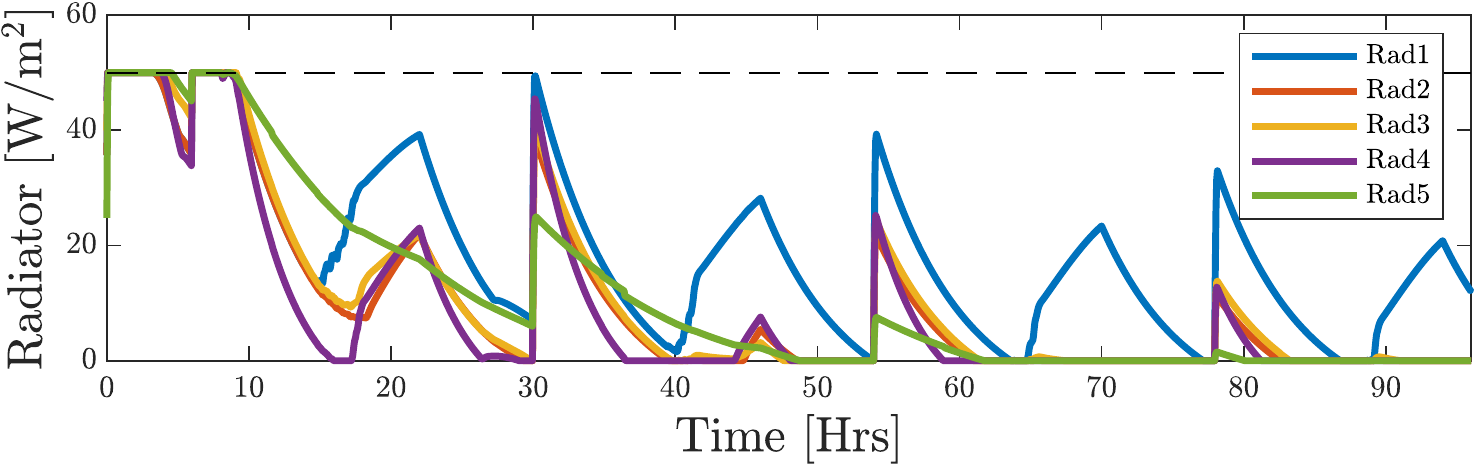}
    \includegraphics[width=\columnwidth]{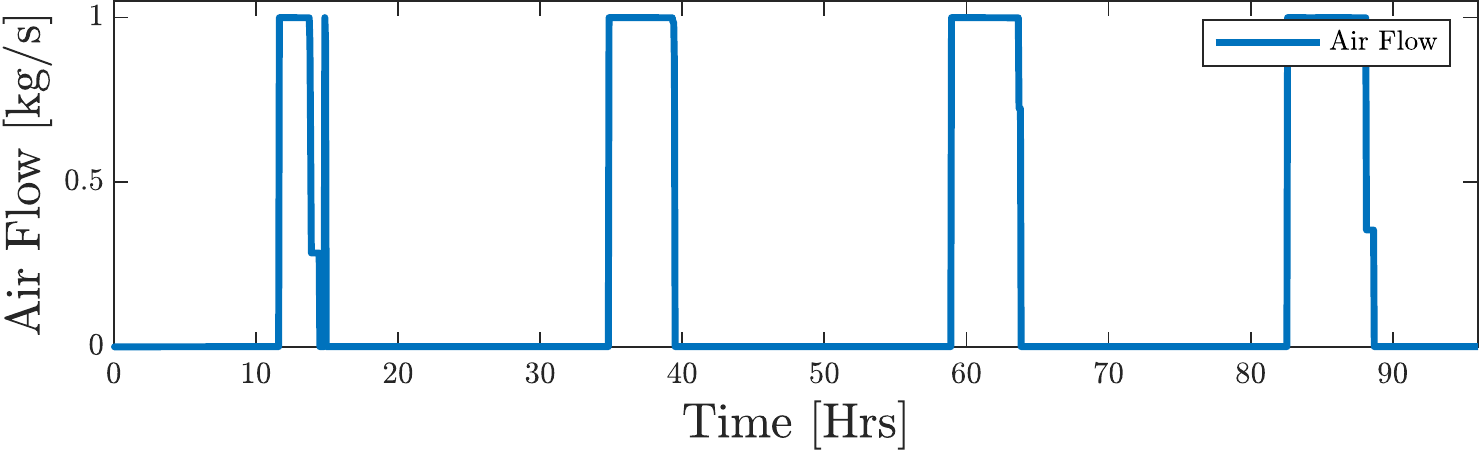}
  \vspace{0.25cm}
  \caption{Simulation of the feedback optimization algorithm \protect\eqref{eq:PGM} on the building dynamics~\eqref{eq:1}.}
  \label{fig:feedback_opt_results}
\end{figure}



In this section, we illustrate how FES can be applied to smart building automation. Consider the 5-room single-story office building in Fig.~\ref{fig:building}. Its dynamics\footnote{The model is generated using the BRCM toolbox~\cite{sturzenegger2014brcm}. 

} can be expressed as,
\begin{equation}\label{eq:1}
    \dot{x} = Ax + B_u u + B_w w + \sum_{i=1}^{n_u} \left(B_{wu,i}w + B_{xu,i}x \right) u_{i}.
\end{equation}
The state $x\in \mc{L}^{113}$ contains the temperatures of the rooms, as well as all other building elements considered by the model, such as wall layers, floor layers, etc. The control inputs $u\in \mc{L}^{8}$ are an air handling unit (AHU) and one radiator in each room. The radiators emit a heat gain between 0 and 50 W/m$^2$, and the AHU inputs are the air flow (0-1kg/s), heating power (0-1000W), and cooling power (0-100W).
We denote the set of these actuation constraints as $\mathcal{U}$.

The disturbances $w\in \mc{L}^{10}$ include the solar radiation, the ambient outdoor air temperature, the temperature of the ground, and the internal heat gains coming from building occupants. The measurement $y\in \mc{L}^{7}$ contains the temperatures of each room as well as the outside air and ground temperatures. Solar radiation and the heat emitted by the buildings occupants are unmeasured, the former since accurately modeling the physics of windows to yield such a measurement is expensive, and the latter due to privacy.

We model 15 building occupants by Markov chains with a time-dependent probability of being in a given room~\cite{wang2011novel}, and consider an 8 hour work day with a 1.5 hour lunch break. The solar radiation and ambient temperature are modelled via periodic functions yielding temperatures and solar gains representative of central European springtime.

The objective is to minimize energy usage while maintaining the room temperatures within a comfortable range $\mc{T} = [T_{\min}~T_{\max}]$. This is encoded in the following cost
\begin{align} 
  \varphi_1(u,y) &= \frac{1}{2} (Hu + c)^Tu 
  + \frac{\eta}{2} \sum_{i=1}^{5} {\mathrm{dist}(y_i,\mc{T})}^2 ,\label{eq:2}
\end{align}
where $H \succ 0$ and $c$ model the electricity price, $\eta>0$ and $\mathrm{dist}(y,\mc{T}) = \min_{u\in \mc{T}} \|y - u\|$. The actuation limits are modelled by the additional nonsmooth term $\varphi_2(u) = \iota_\mc{U}(u)$. 

Based on these costs and the steady-state map associated with \eqref{eq:1}, we construct a FO controller using the algorithm described in Sections~\ref{ex:opt_zeros} and \ref{ex:gradient_descent}.
Finally, we set\footnote{A script is available in the \texttt{gitlab} repository\footref{gitlab} to show that wrong choices of the parameters may yield to instability of the interconnection.} $\epsilon=1$ and $\tau=0.05$.
Simulation results are presented in Figure~\ref{fig:feedback_opt_results}.
The proposed controls are able to keep the rooms between the temperature bounds, with only minor constraint violations.
We compare the controller performance to a hysteresis-based thermostat controller, which turns the radiators and AHU heater on when $T_{\text{room}} \leq \frac{T_{\min} + T_{\max}}{2} - 2$ and off when $T_{\text{room}} \geq \frac{T_{\min} + T_{\max}}{2}$. The AHU cooler is turned on when $T_{\text{room}} \geq \frac{T_{\min} + T_{\max}}{2}+2$, and off when $T_{\text{room}} \leq \frac{T_{\min} + T_{\max}}{2}$, with all temperatures in $^\circ$C.
In our example, the FO method provides a 55.15\% reduction in constraint violations, and a 17.99\% reduction in total cost as measured by~\eqref{eq:2}.

\subsection{FES of Robotic Swarms with Connectivity Objectives}  \label{ss:numerical_example2}
In this example, modified from \cite{krilavsevic2021learning}, we consider a collection of $N = 4$ robotic agents, each tasked with investigating a signal at $\bar r_i^T = [\bar a_i~~\bar b_i]$. Each agent has the unicycle dynamics
\begin{align}
\dot{a}_i = v_i \cos(\theta_i),~~\dot{b}_i = v_i \sin(\theta_i),~~\dot{\theta}_i = p_i,
\end{align}
where $a_i$ and $b_i$ are the coordinates in the Cartesian plane, $\theta$ is the heading angle, and $v_i$, $p_i$ are control inputs, and measures the relative position of the signal and of the other agents, i.e., $y_i = [r_i \!-\! \bar r_i,~[r_i \!-\! r_j]_{j\in \mc{I}\setminus\{i\}}]$ where $r_i^T = [a_i ~~ b_i]$. Each agent implements the (local) control policy \cite{lee2000stable}
\begin{subequations}
\begin{gather}
    v_i = k_1 \|r_i-u_i\| \cos(\phi_i),\\
    p_i = -k_1 \cos(\phi_i)\sin(\phi_i) - k_2 \phi_i,
\end{gather}
\end{subequations}
where $u_i\in \reals^2$ is the position command (i.e., the control input), the heading error is $\phi_i = \pi + \theta_i - \arctan(\frac{b_i-\bar b_i}{a_i-\bar a_i})$, and $k_1 = 1$ and $k_2 = 0.5$ are gains. This control policy stabilizes each agent's dynamics and results in the steady-state maps $h_i(u_i) = u_i$ and the system wide map $h = [h_i]_{i\in \mc{I}} = I$.

The goal of each agent is to approach and inspect its assigned signal, however it is also desirable to maintain a certain degree of connectivity in case of a critical failure. These competing objectives are encoded in the cost function
\begin{equation}
    J_i(u_i,y) = \|u_i - \bar r_i\|^2 + 0.25 \sum_{j \in \mc{I} \setminus \{i\}} \|r_i - r_j\|^2,
\end{equation}
resulting in a strongly monotone game of the same form as in Section~\ref{ex:game_zeros}
\begin{equation} \label{eq:Game_robot}
    \forall i \in \mc{I}~:~\min_{u_i\in \mc{U}_i}~\left\{J_i(u_i,y)~\big|~y = h(u)\right\},
\end{equation}
where $\mc{U}_i = [-5~10] \times [-6~6]$ is the safe traversing area.

A numerical simulation of FES applied to the game described above, and implemented using the best-response algorithm in Section~\ref{ex:best_response} with parameters $\tau = 0.5$, $\varepsilon = 1$, is illustrated in Figures~\ref{fig:box_coord} and \ref{fig:line_coord}. As predicted by Theorem~\ref{thm:sampled_data_stability}, the coupled system is stable and the agents converge to the unique NE of \eqref{eq:Game_robot}.
\begin{figure}
    \centering
    \includegraphics[width=\columnwidth]{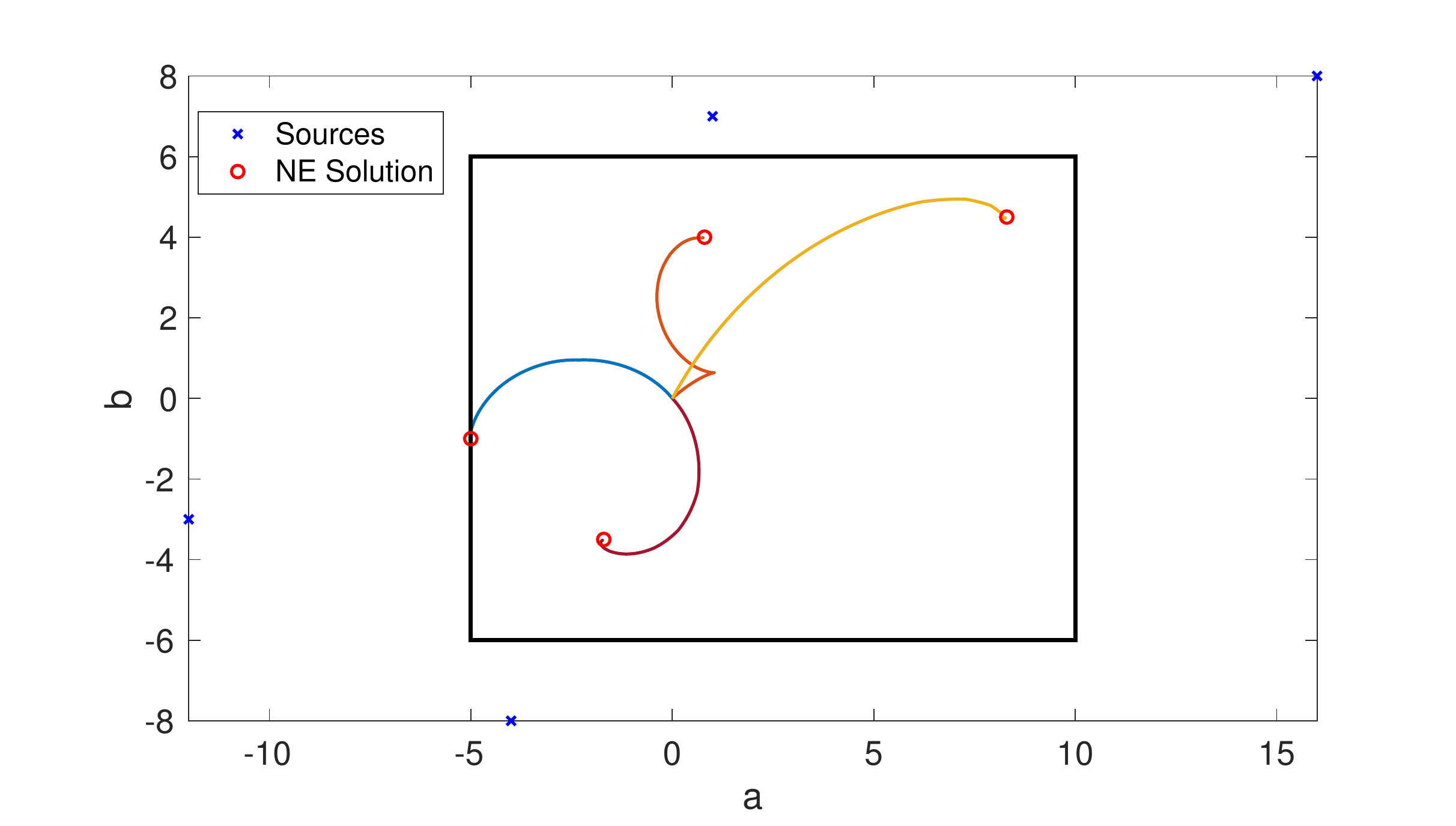}
    \vspace{1mm}
    \caption{Trajectories in the $a$-$b$ plane. The robots reach the NE which differs from the source locations due to the connectivity objective.}
    \label{fig:box_coord}
\end{figure}%
\begin{figure}[htbp]
    \centering
    \includegraphics[width=\columnwidth]{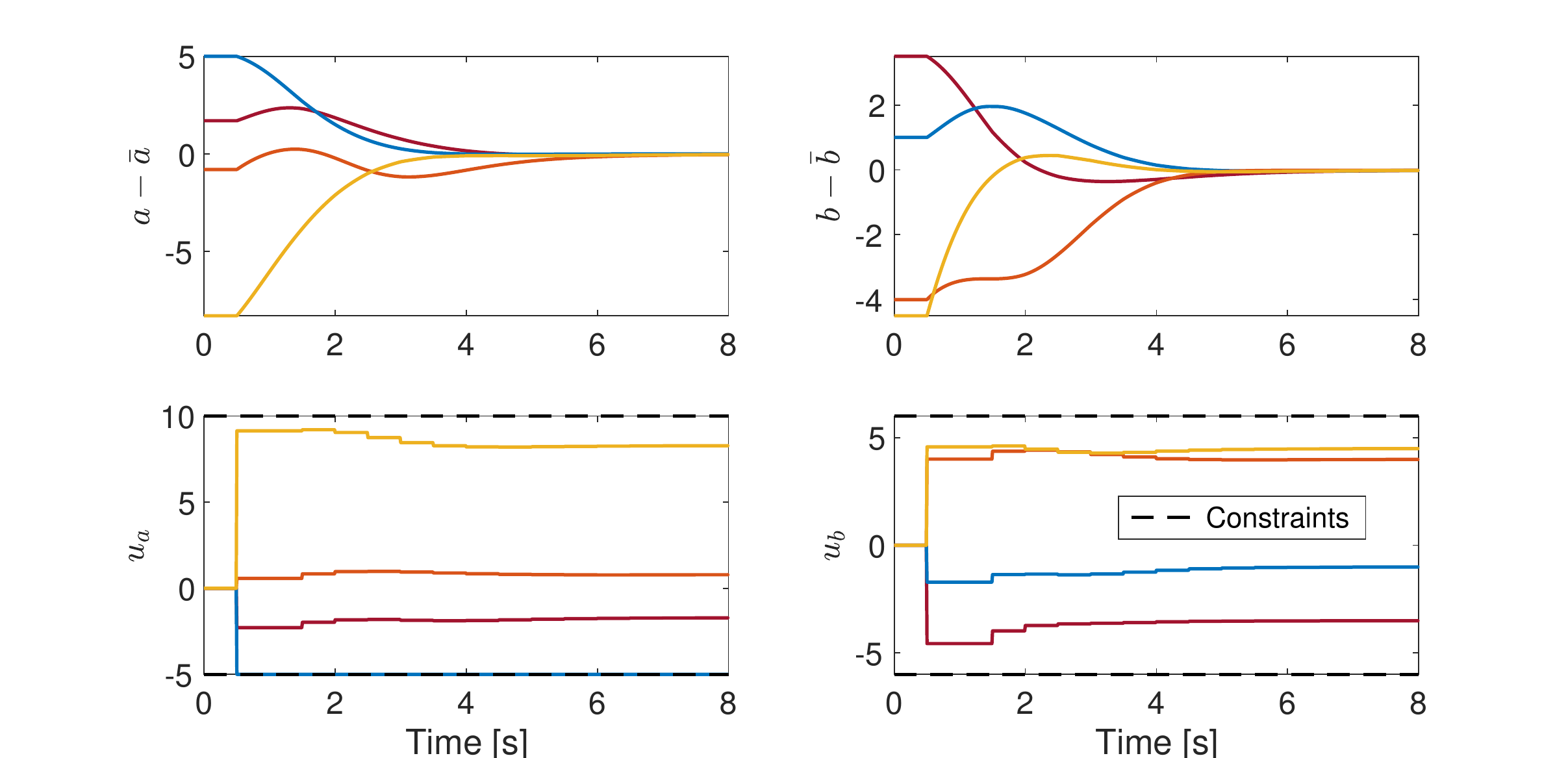}
    \vspace{1mm}
    \caption{State error (top row) and input (bottom row) trajectories for the robotic agents. There is no obvious timescale separation between the algorithm and the physical dynamics and the NE is a stable equilibrium point of the closed-loop system.}
    \label{fig:line_coord}
\end{figure}

\section{Conclusion}
This paper introduced feedback equilibrium seeking (FES), a methodology for guiding dynamical systems to economic equilibria. We provided illustrative examples in robotic coordination and building control and derived conditions for input-to-state stability of the sampled-data closed-loop system. Since feedback optimization (FO) is a special case of FES, this paper also provides the first closed-loop stability analysis of FO in a sampled-data setting. Future research directions include expanding the class of algorithms, handling local stability/convergence, incorporating model learning, real-time constraint satisfaction and tackling incentive design problems using Stackelberg games.

\appendix
%
%
\subsection{Proof of Lemma~\ref{th:W}} \label{proof:W}
By virtue of Assumption~\ref{ass:ctime_lyapunov}, we have that
\begin{equation} \label{eq:diff_ineq}
    \dot V(t) \leq -\mu V(t) + \sigma_{\text c}(\|\dot w(t)\|),
\end{equation}
for any fixed $u \in \mc{U}$. Given the initial condition $V(t_0) = V_0$, it can be easily verified that \eqref{eq:diff_ineq} implies that
\begin{align}
    V(x(t),u,w(t)) &\leq V_0 e^{-\mu(t-t_0)} + \int_{t_0}^t \sigma_{\text c}(\|\dot w(s)\|)~ds\\
    &\leq V_0 e^{-\mu(t-t_0)} + (t-t_0) \sigma_{\text c}(z^0)
\end{align}
where $z^0 = z(t^0) = \sup_{s\in [t^0,t]} \|\dot w(s)\|$. Letting $t = t^{k+1}$ and $t^0 = t^k$, then we obtain that
\begin{equation} \label{eq:Wproof1}
    V(x^{k+1},u,w^{k+1}) \leq e^{-\mu \tau} V(x^k,u,w^k) + \tau \sigma_{\text c}(z^k).
\end{equation}
Now, let $u = u^{k+1}$, then
\begin{align*}
V^{k+1} & \leq e^{-\mu \tau} V(x^k,u^{k+1},w^k) + \tau \sigma_c(z^k)\\
& \leq e^{-\mu\tau} \alpha_2 \| x^k - x_{ss}(u^{k+1},w^k)\|^2 + \tau \sigma_c(z^k)\\
& = e^{-\mu\tau} \alpha_2 (\| x^k - x_{ss}(u^{k},w^k)\| \\ &+ \|x_{ss}(u^{k+1},w^k) - x_{ss}(u^{k},w^k)\|)^2 + \tau \sigma_c(z^k),
\end{align*}
where $\alpha_2$ is from the upper bound on the Lyapunov function, see Assumption \eqref{ass:ctime_lyapunov}~(i). Taking the square root of both sides and using the lower bound on the Lyapunov function, $W = \sqrt V$, and the $\ell_{\text x}$-Lipschitz continuity of $x_{ss}$ w.r.t. $u$, we obtain
\begin{align*} 
W^{k+1} & 
\textstyle
\leq e^{-\frac{\mu\tau}{2}}  \sqrt{(\alpha_2 / \alpha_1)}\, W^k + e^{-\frac{\mu\tau}{2}} \sqrt{\alpha_2} \\&  \|x_{ss}(u^{k+1},w^k) - x_{ss}(u^{k},w^k)\| + \sqrt{\tau}\sigma(z^k)\\
& \leq c_{\text W} W^k + e^{-\frac{\mu\tau}{2}} \sqrt{\alpha_2} \ell_{\text x}\|u^{k+1} - u^k\| + \sqrt{\tau}\sigma(z^k)\\
& = c_{\text W} W^k + c_{\text W} \ell_{\text W} \|u^{k+1} - u^k\|  +  \sqrt{\tau} \sigma(z^k),
\end{align*}
where $\ell_x>0$ is the Lipschitz constant of $x_{ss}$ w.r.t. $u$,  which exists since $x_{ss}$ is continuously differentiable and $\mc U $ compact, $c_{\text W}^2 = \frac{\alpha_2}{\alpha_1} e^{-\mu \tau}$, $\ell_{\text W} = \sqrt{\alpha_1} \ell_{\text x}$ and $\sigma = \sqrt{\sigma_{\text c}}\in \K$.
{\hfill $\blacksquare$}

\subsection{Proof of Lemma~\ref{lem:ParamSel}} \label{proof:ParamSel}
For all $\varepsilon \in (0,1]$, $M$ in \eqref{eq:M} is a positive matrix. Hence, by the Perron--Frobenius theorem, $\rho(M)=\eigmax{M}  \in \R_{+}$. The connection between $\eigmax{M} \in (0,1)$ and the bounds in \eqref{eq:small_gain} follows by a direct computation of $\eigmax{M}$.
{\hfill $\blacksquare$}

\subsection{Proof of Theorem \ref{th:ATE}}
\label{proof:ATE}
Define the mappings $T_{\varepsilon}(u,y):= (1\!-\!\varepsilon)u + \varepsilon T(u,y)$ and $\Tilde T_{\varepsilon}(u,w) := T_{\varepsilon}(u,h(u,w))$. Under Assumption \ref{ass:optRoutine}~(ii), it follows that $\Tilde T_{\varepsilon}(\cdot,w)$ is $\phi(\varepsilon)$-linearly convergent to $u^*(w)$, with $\phi(\varepsilon):=1-\varepsilon(1-c_{\text T}) \in (0,1)$, whenever $\varepsilon \in (0,1]$. Now, consider the weighted input tracking error $\|\delta u^{k+1}\|_P=\| u^{k+1} - u^*(w^{k+1}) \|_P  $; the following inequalities hold:
{\small
\begin{align} \nonumber 
    &\|\delta u^{k+1}\|_P = {\| T_{\varepsilon}( u^k,y^{k+1}) - u^*(w^{k+1})\|}_P \\
    \nonumber
    & \overset{\text{(s.1)}}{\leq} {\|\Tilde T_{\varepsilon}(u^k, w^{k+1}) - u^*(w^{k+1})  \|}_P
        \\
        \nonumber
    & \qquad   + {\| T_{\varepsilon}( u^k,y^{k+1}) -  T_{\varepsilon}(u^k,h(u^k,w^{k+1})) \|}_P \\
    \nonumber
    & \overset{\text{(s.2)}}{\leq} {\phi(\varepsilon)\| u^k - u^*(w^{k+1}) \|}_P + \varepsilon \ell_{\text T} {\| y^{k+1} - h(u^k,w^{k+1}) \|}_P\\
    \nonumber
    & \overset{\text{(s.3)}}{\leq} \phi(\varepsilon) {\| u^k - u^*(w^k) \|}_P 
    +   \phi(\varepsilon)  \ell_{u^*} \|P\| \|w^{k+1}-w^k \|\\
    \nonumber & \qquad     
         + \varepsilon \ell_{\text T} \|P\| \ell_{g} \| x^{k+1}  - x_{\text{ss}}(u^k,w^{k+1}) \| \\
   &\overset{\text{(s.4)}}{\leq} \nonumber
     \phi(\varepsilon) {\| u^k - u^*(w^k) \|}_P +  \phi(\varepsilon) \ell_{u^*} \|P\|  \tau \, z^k  \\
   & \qquad  \textstyle  \nonumber 
         + \varepsilon \frac{\ell_{\text T} \|P\|  \ell_{g}  }{\sqrt{\alpha_1}}\, W(x^{k+1},u^k,w^{k+1})\\
   &\overset{\text{(s.5)}}{\leq} \nonumber
     \phi(\varepsilon) \big( \| u^k - u^*(w^k) \|_P + \|P\|  \ell_{u^*}  \tau  z^k \big), \\
   & \qquad  \textstyle + \varepsilon \frac{\ell_{\text T} \|P\|  \ell_{g}}{\sqrt{\alpha_1}}\, \left( c_{\text W} W(x^{k},u^k,w^{k}) + \sqrt{\tau}\sigma(z^k) \right)
    \label{eq:dis-u}
\end{align}
}
Here, (s.1) follows by adding and subtracting $\Tilde T_{\varepsilon}(u^k, w^{k+1})$, or, equivalently, $T_{\varepsilon}(u^k,h(u,w))$, and the triangular inequality; (s.2) by the $\phi(\varepsilon)$-linear convergence of $\Tilde T_{\varepsilon}$ and the Lipschitz continuity of $T$ (Assumptions \ref{ass:optRoutine}); the first line in (s.3) by adding and subtracting $u^*(w^k)$, the triangular inequality, and the Lipschitz continuity of the solution mapping $u^*$ (Assumption \ref{ass:existence}), while the second by the $\ell_g$-Lipschitz continuity of $g$ (Assumption \ref{ass:sys}); (s.4) since $\|w^{k+1}-w^k \| \leq \tau z^k$, while the second line is due to Assumption \ref{ass:ctime_lyapunov}~(i) and since $W = \sqrt{V}$; finally, (s.5) follows by applying the square root of \eqref{eq:Wproof1} on $W(x^{k+1},u^k,w^{k+1})$.

By Lemma \ref{th:W}, $W^k=W(x^{k},u^k,w^{k})$ satisfies
\begin{align} \label{eq:Wproof2}
  & W^{k+1} \leq c_{\text W} W^k + c_{\text W} \ell_{\text W} \|u^{k+1} - u^k\| + \sqrt{\tau} \sigma(z^{k}).
\end{align}
Next, we derive a bound for $\|u^{k+1} - u^k\|$ in \eqref{eq:Wproof2}, i.e.,
{\small
\begin{align}
\nonumber
  &  \|u^{k+1} - u^k\| = \|T_{\varepsilon}(u^k,y^{k+1})-u^k \| = \varepsilon\|T(u^k,y^{k+1})-u^k  \| \\
\nonumber
        & \overset{\text{(s.1)}}{\leq} \varepsilon\, {\|  T(u^k,h(u^k,w^{k+1}))-u^k\|}\\
\nonumber
        &\qquad +  \varepsilon  \,  \| T(u^k,y^{k+1})- T(u^k,h(u^k,w^{k+1}))\|\\
        \nonumber
       &\overset{\text{(s.2)}}{\leq} \varepsilon  \big({\| \tilde T(u^k,w^{k+1})\!-\!u^*(w^{k+1})\|} + {\|u^k\!-\!u^*(w^k)\|}  \\
        \nonumber
        &\qquad  
        + \|u^*(w^{k+1})\!-\!u^*(w^k) \|\big) +\varepsilon  \ell_{\text T}\,   \| y^{k+1}- h(u^k,w^{k+1})\|\\
\nonumber
       &\overset{\text{(s.3)}}{\leq} \varepsilon \| P^{-1}\| (1+c_{\text T}) \, ({\|  u^k-u^*(w^k)\|}_P + \|P\| \ell_{u^*} \tau   z^k ) \\
        \nonumber
        &\qquad  +\varepsilon  \ell_{\text T}\,   \| y^{k+1}- h(u^k,w^{k+1})\| \\
\nonumber
        &\overset{\text{(s.4)}}{\leq} \varepsilon \| P^{-1}\| (1+c_{\text T}) \, ({\|  u^k-u^*(w^k)\|}_P + \|P\| \ell_{u^*} \tau   z^k  ) \\
        \label{eq:deltaUproof2}
        &\qquad \textstyle +  \frac{\ell_{\text T}\, \ell_{\text g}  }{\sqrt{\alpha_1}} \varepsilon \,  \left( c_{\text W} W(x^{k},u^k,w^{k}) + \sqrt{\tau}\sigma(z^k) \right).
\end{align}}%
Here, (s.1) follows by adding and subtracting $\tilde T(u^k,w^{k+1})$, and the triangular inequality; the first term in (s.2) by adding and subtracting $u^*(w^{k+1})$ and $u^*(w^k)$ and the triangular inequality, while the second  by the Lipschitz continuity of $ T$ (Assumption \ref{ass:T_lipschitz}); (s.3) by the facts: $\lambda_{\min}(P) \|a \| \leq \|a \|_P$, $\forall a \in \R^n$, and $ \lambda_{\min}(P) = (\|P^{-1}\|)^{-1}$, the linear convergence of $\tilde T$ (Assumption \ref{ass:optRoutine}~(ii)) and the Lipschitz continuity of $u^*$ (Assumptions \ref{ass:existence}); (s.4) by the Lipschitz continuity of $g$ and the same steps in the last inequality of \eqref{eq:dis-u}.

By substituting \eqref{eq:deltaUproof2} into \eqref{eq:Wproof2}, we obtain
\begin{align}     \nonumber
     W^{k+1} & \textstyle \leq  \left(1 +  \frac{ \ell_{\text W} \ell_{\text T} \ell_{\text g}}{\sqrt{\alpha_1}} \varepsilon c_{\text W} \right) \left( c_{\text W} W^k + \sqrt{\tau}\sigma(z^{k}) \right) \\
        \nonumber
        &\quad 
         +   c_{\text W} \ell_{\text W} \varepsilon \| P^{-1}\| (1+c_{\text T}) \big( \|  u^k-u^*(w^k)\|_P  \\
          & \quad
          + \|P\| \ell_{u^*} \tau \,   z^k \big).
    \label{eq:dis-e}
\end{align}
Now, let us define the matrix $M$ as in \eqref{eq:M}, and the vectors
\begin{align}
\varpi^k := \left[
    \begin{smallmatrix}
    {\|\delta u^k\|}_P\\
     W^k 
     \end{smallmatrix}
    \right],
\quad
     q(z^k) :=  
    \left[
    \begin{smallmatrix}
    \|P\| \ell_{u^*} \tau \, z^k  \\
(c_{\text W})^{-1} \sqrt{\tau}\sigma(z^k)
    \end{smallmatrix}
    \right]. 
\end{align}
The inequalities \eqref{eq:dis-u} and \eqref{eq:dis-e} can be compactly recast as
\begin{align}
\label{eq:McI}
\varpi^{k+1}
   &\leq 
M (
\varpi^k
 +
q(z^k)).
\end{align}
By telescoping \eqref{eq:McI}, we get 
\begin{align}\textstyle
\label{eq:McI_2}
\varpi^{k+1}
   \leq 
{(M)}^{k+1} \varpi^0 
 +
\sum_{s=0}^{k} {(M)}^{k+1-s}  q(z^s).
\end{align}
Recall that $M$ is a Schur matrix, i.e., $\rho(M ) < 1$, whenever $\tau$ and $\varepsilon$ satisfy \eqref{eq:small_gain}, by Lemma~\ref{lem:ParamSel}. For such a matrix, there exist constants $r \in \R_{+}$ and $c_{\text M} \in [0,1)$ s.t. $\|{(M)}^k \| \leq r  {(c_{\text M})}^k$, see e.g. \cite[Ex.~3.4]{jiang2001input}. Hence, by taking the norm of \eqref{eq:McI_2}, on both sides, and using the previous result, we get
\begin{align*} 
   \|\varpi^{k+1}\|
   \leq 
r \,{(c_{\text M})}^{k+1} \|\varpi^{0}\|
+ r\, \sum_{s=0}^{k} {(c_{\text M})}^{k+1-s}  \|q(z^s)\| .
\end{align*}
If we let $\bar z^k = \sup_{0 \leq s \leq k} \|z^s\|$, then we have that
\begin{align*}\textstyle
\| \varpi^{k+1}\|
 \leq r \,
{\left( c_{\text M} \right)}^{k+1} \|\varpi^{0}\|
+ r \, c_{\text M}\left( \sum_{s=0}^{k} {\left( c_{\text M} \right)}^s \right) \|q(\bar z^k)\|,
\end{align*}
where last term into parentheses is the first $s\!+\!1$ addends of a geometric series. Hence, it can be upper-bounded, yielding
\begin{align} 
\label{eq:ISS3}
 \textstyle
\|\varpi^{k+1}\|
\leq r \, {\left( c_{\text M} \right)}^{k+1} \|\varpi^{0}\|
+  r \, \frac{c_{\text M}}{1- c_{\text M}} \, \|q(\bar z^k)\|.
\end{align}
Now, by using $\lambda_{\min}(P){\|\cdot\|}_P \leq \|\cdot \| \leq \lambda_{\max}(P){\|\cdot\|}_P$ and the Lyapunov bounds in Assumption~\ref{ass:ctime_lyapunov}~(i), we can write
\begin{align}
\label{eq:boundsVarpi}
m_1 \left\| \left[ 
 \begin{matrix}
        \delta x^{k}\\
        \delta u^{k}
    \end{matrix} 
    \right]
    \right\| 
    \leq
    \|\varpi^k\|
    \leq
    m_2 \left\| \left[ 
 \begin{matrix}
        \delta x^{k}\\
        \delta u^{k}
    \end{matrix} 
    \right]
    \right\|,
\end{align}
$m_1\!=\! \min\{\lambda_{\min}(P),\sqrt{\alpha_1} \}$,  $m_2\!=\! \max\{\lambda_{\max}(P),\sqrt{\alpha_2} \}.$
Finally, by using the bounds in \eqref{eq:boundsVarpi} into \eqref{eq:ISS3}, we obtain the ISS inequality in \eqref{eq:thISS}, with $\eta_1 = r \frac{ m_2}{m_1} $ and $\eta_2 = \frac{r}{m_1}$.
{\hfill $\blacksquare$}

\subsection{Proof of Corollary \ref{cor:IOg}}
\label{proof:CorIOg}
The input-output tracking error $ \|\delta y^k \|=\|y^k - y^*(w^k)\|$, where $y^*(w^k)=h(u^*(w^k),w^k)$, satisfies the following:
\begin{align}
\nonumber
    {\|y^k - h(u^*(w^k),w^k)\|} 
     & \leq \ell_{g} {\|x^k-x_{\text{ss}}(u^*(w^k),w^k) \|} \\
   \nonumber 
                          & \leq \ell_{g}(\| \delta x^k \| +\ell_{\text x}\| \delta u^k \|)\\
                          & \leq
                            \ell_{g}(1 +\ell_{\text x}) \left\| \left[ 
 \begin{smallmatrix}
        \delta x^{k}\\
        \delta u^{k}
    \end{smallmatrix} 
    \right]
    \right\| .
\label{eq:AGproof}
\end{align}
The first inequality follows by the $\ell_g$-Lipschitz continuity of $g$ (Assumption \ref{ass:sys}); the second by adding and subtracting $x_{\text{ss}}(u^k,w^k)$, the triangular inequality and the $\ell_{\text x}$-Lipschitz continuity of $x_{\text{ss}}$; the last since $\|\delta x^k\|, \|\delta u^k\| \leq
\sqrt{\| \delta x^{k}\|^2+\| \delta u^{k}\|^2}
 .$
Finally, by substituting the ISS bound \eqref{eq:thISS} in \eqref{eq:AGproof}, and taking the limit for $k \rightarrow \infty$, we obtain the asymptotic gain in \eqref{eq:AIOG}.
{\hfill $\blacksquare$}
\bibliographystyle{ieeetr}
\bibliography{library}

\end{document}